\documentclass{llncs}
\usepackage{amssymb, amsmath}
\usepackage{graphicx,color}
\usepackage{epsfig,graphics,epic,eepic,stmaryrd,url}
\usepackage{algorithmic}
\usepackage[section]{algorithm}
\begin{document}
\pagestyle{plain}
\title{Bipartite powers of $k$-chordal graphs}
\author{L. Sunil Chandran\and Rogers Mathew}
\institute{Department of Computer Science and Automation, \\ 
Indian Institute of Science, Bangalore -- 560012, India.\\
\texttt{\{sunil,rogers\}@csa.iisc.ernet.in}}
\maketitle
\bibliographystyle{plain}
\def\remarkname{Observation}
\newcounter{line}
\newcommand{\MYSTATE}[1]{\STATE{\refstepcounter{line}\theline. #1}}
\newcommand{\bbox}{\rule{0.6em}{0.6em}}

\begin{abstract}
Let $k$ be an integer and $k \geq 3$. A graph $G$ is \emph{$k$-chordal} if $G$ does not have an induced cycle of length greater than $k$. From the definition it is clear that $3$-chordal graphs are precisely the class of chordal graphs. 
Duchet proved that, for every positive integer $m$, if $G^m$ is chordal then so is $G^{m+2}$. Brandst\"adt et al. in [Andreas Brandst\"adt, Van Bang Le, and Thomas Szymczak. Duchet-type theorems for powers of HHD-free graphs. \emph{Discrete Mathematics}, 177(1-3):9-16, 1997.] showed that if $G^m$ 
is $k$-chordal, then so is $G^{m+2}$.

Powering a bipartite graph does not preserve its bipartitedness. In order to preserve the bipartitedness of a bipartite graph while powering Chandran et al. 
introduced the notion of \emph{bipartite powering}. This notion was introduced to aid their study of boxicity of chordal bipartite graphs. Given a bipartite graph $G$ and an odd positive integer $m$, we define the graph $G^{[m]}$ to be a bipartite graph with $V(G^{[m]})=V(G)$ and $E(G^{[m]})=\{(u,v)
~|~ u,v \in V(G), d_G(u,v)\mbox{ is odd, and } d_G(u,v)\leq m\}$. The graph $G^{[m]}$ is called the \emph{$m$-th bipartite power} of $G$. 

In this paper we show that, given a bipartite graph $G$, if $G$ is $k$-chordal then so is $G^{[m]}$, where $k$, $m$ are positive integers such that $k\geq 4$ and $m$ is odd.  

\medskip\noindent\textbf{Key words: }
$k$-chordal graph, hole, chordality, graph power, bipartite power.
\end{abstract}
\section{Introduction}
A \emph{hole} is a chordless (or an induced) cycle in a graph. The \emph{chordality} of a graph $G$, denoted by $\mathcal{C}(G)$, is defined to be the size of a largest hole in $G$, if there exists a cycle in $G$. If $G$ is acyclic, then its chordality is taken as $0$. A graph $G$ is \emph{$k$-chordal} if $\mathcal{C}(G) \leq k$. In other words, a graph is $k$-chordal if it has no holes with more than $k$ vertices in it. Chordal graphs are exactly the class of $3$-chordal graphs and chordal bipartite graphs are bipartite, $4$-chordal graphs.  
$k$-chordal graphs 
have been studied in the literature in \cite{BodThil}, \cite{Sun5}, \cite{CSL}, \cite{YonDour}, 
\cite{Dragan20051} and \cite{Spinrad1991227}. For example, Chandran and Ram \cite{Sun5}
proved that the number of minimum cuts in a $k$-chordal graph is at most $\frac{(k+1)n}{2}-k$.  
Spinrad\cite{Spinrad1991227} showed that $(k-1)$-chordal graphs can be recognized in 
$O(n^{k-3}M)$ time, where $M$ is the time required to multiply two $n$ by $n$ matrices.

Powering and its effects on the chordality of a graph has been a topic of interest. 
The $m$-th power of a graph $G$, denoted by $G^m$, is a graph with vertex set
$V(G^m) = V(G)$ and edge set $E(G^m)=\{(u,v)~|~u \neq v, d_G(u,v) \leq m\}$, where 
$d_G(u,v)$ represents the distance between $u$ and $v$ in $G$. 
Balakrishnan and Paulraja \cite{BalPaul1} proved that odd powers of chordal graphs are chordal.
Chang and Nemhauser \cite{changNemhauser} showed that if $G$ and $G^2$ are chordal then so are
all powers of $G$. Duchet \cite{Duchet} proved a stronger result which says that if $G^m$ is 
chordal then so is $G^{m+2}$. Brandst\"adt et al. in \cite{BrandPower} showed that if $G^m$ is $k$-chordal then so is $G^{m+2}$, where $k\geq 3$ is an integer. Studies on families of graphs that are closed under powering can also be seen in the literature. For instance, it is known that interval graphs, 
proper interval graphs \cite{Raychauduri2}, strongly chordal graphs \cite{Lub1}, circular-arc graphs \cite{Raychauduri}\cite{Flotow3}, cocomparability 
graphs \cite{Flotow2} etc. are closed under taking powers. 

Subclasses of bipartite graphs, like chordal bipartite graphs, are not closed under powering since the $m$-th power of a bipartite graph need not be even bipartite. Chandran et al. in \cite{SunMatRog} introduced the notion of \emph{bipartite powering} to retain the 
bipartitedness of a bipartite graph while taking power. 
Given a bipartite graph $G$ and an odd positive integer $m$, $G^{[m]}$ is a bipartite graph with $V\left(G^{[m]}\right)=V(G)$ and
$E\left(G^{[m]}\right)=\{(u,v) ~|~ u,v\in V(G), d_G(u,v)\mbox{ is odd, and } 
d_G(u,v)\leq m\}$. The graph $G^{[m]}$ is called the \emph{$m$-th bipartite power} of $G$. It was shown in \cite{SunMatRog} that, for every odd positive integer $m$, the $m$-th bipartite power of a tree is chordal bipartite. The intention there was to construct chordal bipartite graphs of high boxicity. The fact that the chordal bipartite graph under consideration was obtained as a  bipartite power of a tree was crucial for proving that its boxicity was high. Since trees are a subclass of chordal bipartite graphs, a natural question that came up was the following: is it true that the $m$-th bipartite power of  every chordal bipartite graph is chordal bipartite? In this paper we answer this question in the affirmative. In fact, we prove a more general result. 
\subsubsection{Our Result}
Let $m$, $k$ be positive integers such that $m$ is odd and $k \geq 4$. Let $G$ be a bipartite graph.   
If $G$ is $k$-chordal, then so is $G^{[m]}$.  Note that the special case when $k=4$ gives us the following result: chordal bipartite graphs are closed under bipartite powering. 

\section{Graph Preliminaries}
Throughout this paper we consider only finite, simple, undirected graphs. 
For a graph $G$, we use $V(G)$ to denote the set of vertices of $G$. Let $E(G)$ denote its edge set. For every $x,y \in V(G)$, $d_G(x,y)$ represents the distance between $x$ and $y$ in $G$. 
For every $u \in V(G)$, $N_G(u)$ denotes its \emph{open neighborhood} in $G$, i.e. $N_G(u) = \{v~|~(u,v)\in E(G)\}$. 
A path $P$ on the vertex set $V(P) = \{v_1,v_2,\ldots,v_n\}$ (where $n\geq 2$) has its edge set $E(P) = \{(v_i,v_{i+1})~|~1 \leq i \leq n-1\}$. 
Such a path is denoted by $v_1 v_2 \ldots v_n$. If $v_i,v_j\in V(P)$, $v_iPv_j$ is the path $v_i v_{i+1}\ldots v_{j}$. 
%
The length of a path $P$ is the number of edges in it and is denoted by $||P||$. 
A cycle $C$ with vertex set $V(C)=\{v_1,v_2,\ldots,v_n\}$, and edge set $E(C)=\{(v_i,v_{i+1}) ~|~1\leq i\leq n-1\}\cup \{(v_n,v_1)\}$ is denoted as $C = v_1 v_2 \ldots v_n v_1$. We use $||C||$ to denote the length of cycle $C$. 

\section{Holes in Bipartite Powers}
Let $H$ be a bipartite graph. Let $\mathcal{B}(H)$ be a family of graphs constructed from $H$ in the following manner: $H' \in \mathcal{B}(H)$ if corresponding to each vertex $v \in V(H)$ there exists a nonempty bag of vertices, say $B_v$, in $H'$ such that (a) for every $x \in B_u$, $y \in B_v$, $(x,y) \in E(H')$ if and only if $(u,v) \in E(H)$, and (b) vertices within each bag in $H'$ are pairwise non-adjacent. Below we list a few observations about $H$ and every $H'$ (, where $H' \in \mathcal{B}(H)$): 
\begin{remark}
\label{observation1}
$H'$ is bipartite. 
\end{remark}
\begin{remark}
\label{observation2}
$H$ is an induced subgraph of $H'$. 
\end{remark}
\begin{remark}
\label{observation3}
Let $k$ be an integer such that $k \geq 4$. If $H$ is $k$-chordal, then so is $H'$. 
\end{remark}
\begin{proof}
Any hole of size greater than $4$ in $H'$ cannot have more than one vertex from the same bag, say $B_v$, as such vertices have the same neighborhood. Hence, the vertices of a hole (of size greater than $4$) in $H'$ belong to different bags and thus there is a corresponding hole of the same size in $H$. 
\end{proof}

\begin{theorem}\label{bippowtheorem}
Let $m$, $k$ be positive integers such that $m$ is odd and $k \geq 4$. Let $G$ be a bipartite graph.   
If $G$ is $k$-chordal, then so is $G^{[m]}$.   
\end{theorem}
\begin{proof}
We prove this by contradiction. Let $p$ denote the size of a largest induced cycle, say $C = u_0 u_1 \ldots  u_{p-1} u_0$, in $G^{[m]}$. 
Assume $p > k$. Then, $p \geq 6$ (since $k \geq 4$ and $G^{[m]}$ is bipartite).
Between each 
$u_{i-1}$ and $u_i$, where $i \in \{0, \ldots, p-1\}$, there exists a shortest path of length not more than $m$ in 
$G$ \footnotemark[1]. 
Let $P_i$ be one such shortest path between $u_{i-1}$ and $u_i$ in $G$. 

Let $H$ be the subgraph induced on the vertex set $\bigcup_{i=0}^{p-1}V(P_i)$ in $G$. As mentioned in the beginning of this section, construct a graph $H'$ from $H$, where $H' \in \mathcal{B}(H)$, in the following manner: for each $v \in V(H)$, let $|B_v| = |\{P_i~|~0\leq i \leq p-1, v \in V(P_i)\}|$ i.e., let $B_v$ have as many vertices as the number of paths in $\{P_0 \ldots P_{p-1}\}$ that share vertex $v$ in $H$. 
For each $i \in \{0, \ldots , p-1\}$, let $Q_i' = u_{i-1}Q_i$ be a shortest path between $u_{i-1}$ and $u_i$ in $H'$ such that no two paths $Q_i$ and $Q_j$ (where $i \neq j$) share a vertex \footnotemark[1] 
\footnotetext[1]{throughout this proof expressions involving subscripts of $u$, $P$, $Q$, and $Q'$ are to be taken modulo $p$. Every such expression should be evaluated to a value in $\{0, \ldots , p-1 \}$. For example, consider a vertex $u_i$, where $i < p$
Then, $p + i = i$.}. 
From our construction of $H'$ from $H$ it is easy to see that such  paths exist. Let $Q_i = v_{i,1} v_{i,2} \ldots v_{i,r_i}u_i$, where $r_i = ||Q_i|| \geq 0$. Thus, $Q_i' = u_{i-1}v_{i,1} v_{i,2} \ldots v_{i,r_i}u_i$. 
Clearly, $||Q_i'|| = ||P_i|| \leq m$. The reader may also note that the cycle $C$ ($ = u_0 u_1 \ldots  u_{p-1} u_0$) which is present in $G^{[m]}$ will be present in $H^{[m]}$ and thereby in $H'^{[m]}$ too.  

In order to prove the theorem, it is enough to show that there exists an induced cycle of size at least $p$ in $H'$. Then by combining Observation \ref{observation3} and the fact that $H$ is an induced subgraph of $G$, we get $k \geq \mathcal{C}(G) \geq \mathcal{C}(H) \geq \mathcal{C}(H') \geq p$ contradicting our assumption that $p > k$. Hence, in the rest of the proof we show that $\mathcal{C}(H') \geq p$.

Consider the following drawing of the 
graph $H'$. Arrange the vertices $u_0, u_1,$  $\ldots, u_{p-1}$ in that order 
on a circle in clockwise order. 
Between each $u_{i-1}$ and $u_{i}$ on  the circle arrange the vertices $v_{i,1}, v_{i,2}, \ldots , v_{i,r_i}$ in that order in clockwise order. Recall that these vertices are the internal vertices of path $Q_i'$. 
\begin{remark}
\label{neighborObv}
In this circular arrangement of vertices of $H'$, each vertex has an edge (in $H'$) with both its left neighbor and right neighbor in the arrangement.  
\end{remark}

Let $x_1, x_2 \in V(H')$, where $x_1 \in V(Q_i)$, $x_2 \in V(Q_j)$. We define the \emph{clockwise distance from $x_1$ to $x_2$}, denoted by $clock\_dist(x_1,x_2)$, as the minimum non-negative integer $s$ such that $j = i + s$. Similarly, the \emph{clockwise distance from $x_2$ to $x_1$}, denoted by $clock\_dist(x_2,x_1)$, is the minimum non-negative integer $s'$ such that $i = j+s'$. Let $x,y,z \in V(H')$.  
We say $y<_x z$ if scanning the vertices of $H'$ in 
clockwise direction along the circle starting from $x$, vertex $y$ is encountered before $z$.  
Let $x \in V(Q_i)$. Vertex $y$ is called the \emph{farthest neighbor of $x$ before $z$} if $y \in N_{H'}(x)$, $y \in V(Q_{i}) \cup V(Q_{i+1}) \cup V(Q_{i+2})$, $y <_x z$, and for every other $w \in N_{H'}(x)$  either $z <_x w$ or $w \notin V(Q_{i}) \cup V(Q_{i+1}) \cup V(Q_{i+2})$ or both. 
\begin{remark}
\label{farthestNeighborObv} 
There always exists a vertex which is the farthest neighbor of $x$ before $z$, unless $(x,z) \in E(H')$ and $z \in V(Q_{i}) \cup V(Q_{i+1}) \cup V(Q_{i+2})$.  
\end{remark}

Let $\{A,B\}$ be the bipartition of the bipartite graph $H'$. We categorize the edges of $H'$ as follows: an edge $(x,y) \in E(H')$ 
is called an \emph{$l$-edge}, if $l = \min(clock\_dist(x,y),clock\_dist(y,x) )$. 
\begin{figure}[H]
\centerline{\input{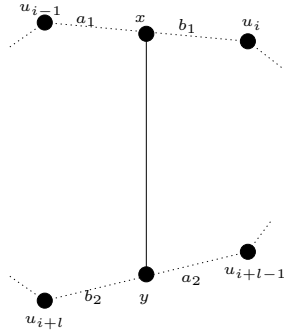}}
\caption{$x \in V(Q_i), y \in V(Q_{i+l})$ and let $(x,y) \in E(H')$ be an $l$-edge, where $l > 2$. The dotted line between $u_{i-1}$ and $u_i$ indicate the path $Q_i$. Similarly, the dotted line between $u_{i+l-1}$ and $u_{i+l}$ indicate the path $Q_{i+l}$.}
\label{l-edgeFig}
\end{figure}

\begin{claim}
\label{Ledgeclaim}
$H'$ cannot have an $l$-edge, where $l>2$.
\end{claim}
\begin{proof}
Suppose $H'$ has an $l$-edge, where $l > 2$, between $x \in Q_i$ and $y \in Q_{i+l}$ (see Fig. \ref{l-edgeFig}). Let $a_1 = ||u_{i-1}Q_i'x||,~  b_1 = ||xQ_i'u_i||,~ a_2 = ||u_{i+l-1}Q_{i+l}'y||$ and $ b_2 = ||yQ_{i+l}'u_{i+l}||$. We consider the following two cases: 
\begin{case}
 $l$ is even.
\end{case}
In this case $u_{i-1}$ and $u_{i+l-1}$ will be on the same side of the bipartite graph $H'$. Without loss of generality, let $u_{i-1}, u_{i+l-1} \in A$. Then, $u_{i}, u_{i+l} \in B$. 
We know that, for every $w_1, w_2 \in V(H'^{[m]})$ with $w_1 \in A$ and $w_2 \in B$, if $(w_1,w_2) \notin E(H'^{[m]})$ then $d_{H'}(w_1,w_2) \geq m+2$ (recalling $m$ and $d_{H'}(w_1,w_2)$ are odd integers).  
Therefore, we have 
$a_1 + 1 + b_2 \geq d_{H'}(u_{i-1},u_{i+l}) \geq m+2$.  
Similarly, $b_1 + 1 + a_2 \geq d_{H'}(u_{i},u_{i+l-1}) \geq m+2$. Summing up the two inequalities we get, $(a_1 + b_1) + (a_2 + b_2) \geq 2m + 2$. This implies that either $||Q_i'||$ or $||Q_{i+l}'||$ is greater  than $m$ which is a contradiction. 
\begin{case}
 $l$ is odd (proof is similar to the above case and hence omitted).
\end{case}
Hence we prove the claim. 
\end{proof}

We find a cycle $C' = z_0z_1\ldots z_qz_0$ in $H'$ using Algorithm $3.1$ \footnotemark[1]. Please read the algorithm before proceeding further.   
\footnotetext[1]{throughout this proof expressions involving subscripts of $z$ are to be taken modulo $q+1$. Every such expression should be evaluated to a value in $\{0, \ldots , q \}$. For example, consider a vertex $z_a$, where $a < q+1$.  Then, $q+1+a = a$.}. 

\begin{algorithm}[H]
\begin{algorithmic}
\label{cycle-construction}
\caption{Finding Cycle $C'$ in $H'$ such that $||C'|| \geq ||C||$}
\setcounter{line}{0}
\MYSTATE{$l \leftarrow \max_{l'}(H' \mbox{has an }l'\mbox{-edge})$. Without loss of generality assume that  this $l$-edge is between a vertex in $Q_0$ and a vertex in $Q_{l}$}
\MYSTATE{Scan the vertices of $Q_0$ in clockwise direction to find the first vertex $z_0$, where $z_0 \in V(Q_0)$, which has an $l$-edge to a vertex in $Q_{l}$.}
\MYSTATE{Scan the vertices of $Q_{l}$ in clockwise direction to find the last vertex in $Q_l$ which is a neighbor of $z_0$ in $H'$. Call it $z_1$.}
\MYSTATE{Find the farthest neighbor of $z_1$ before $z_0$. Call it $z_2$.} /* refer proof of Claim \ref{farthestNeighborClaim} for a proof of existence of such a $z_2$*/
\MYSTATE{$s \leftarrow 2$.}
\WHILE{$(z_s,z_0) \notin E(H')$}
\MYSTATE{Find the farthest neighbor of $z_s$ before $z_0$. Call it $z_{s+1}$.} /* such a neighbor exists by Observation \ref{farthestNeighborObv}*/
\MYSTATE{ $s \leftarrow s + 1$.}
\ENDWHILE
\MYSTATE{ $q \leftarrow s$.}
\MYSTATE{ Return cycle $C' = z_0 z_1\ldots z_{q} z_0$.}
\end{algorithmic}
\end{algorithm}

\begin{claim}
\label{farthestNeighborClaim}
There always exists a farthest neighbor of $z_1$ before $z_0$. 
\end{claim}
\begin{proof}
Note that $z_0 \in Q_0$ and $z_1 \in Q_l$, where $l \leq 2$ (by Claim \ref{Ledgeclaim}). Recalling that $||C|| = p \geq 6$, we have $z_0 \notin  V(Q_{l}) \cup V(Q_{l+1}) \cup V(Q_{l+2})$. Hence  by Observation \ref{farthestNeighborObv}, the claim is true.  
\end{proof}

\begin{claim}
 \label{whileClaim}
The while loop in Algorithm $3.1$ terminates after a finite number of iterations. 
\end{claim}
\begin{proof}
From Observation \ref{neighborObv}, we know that each vertex has an edge (in $H'$) with both its left neighbor and right neighbor in the circular arrangement. Each time when Step $6$ of Algorithm $3.1$ is executed, a vertex $z_{s+1}$ is chosen such that $z_{s+1}$ is the farthest neighbor of $z_s$ before $z_0$. 
Since $H'$ is a finite graph, there will be a point of time in the execution of the algorithm when in Step $6$ it picks a $z_{s+1}$ such that $(z_{s+1}, z_0) \in E(H')$ . 
\end{proof}
From Claim \ref{whileClaim}, we can infer that $C'$ is a cycle. 

\begin{claim}\label{inducedcycleclaim}
$C'$ is an induced cycle in $H'$. 
\end{claim}
\begin{proof}
Suppose $C'$ is not an induced cycle. Then there exists a chord $(z_a,z_b)$ in $C'$. Since $(z_a, z_b)$  is a chord, we have $b \neq a-1$ or $b \neq a+1$.  
Let $l =\max_{l'}(H'$ has an $l'\mbox{-edge})$. 
Let $z_a \in V(Q_i)$, $z_b \in V(Q_j)$. 
We know that $\min (clock\_dist(z_a, z_b),$ $clock\_dist(z_b, z_a)) \leq l$. Without loss of generality, assume $clock\_dist(z_a, z_b) \leq l \leq 2$ (from Claim \ref{Ledgeclaim}). That is, $j-i \leq l \leq 2$ and $(z_a,z_b)$ is a $(j-i)$-edge. If $z_a = z_0$, then $z_b \neq z_1$ and the algorithm exits from the while loop, when $q=b$, thus returning a cycle $z_0 \ldots z_b z_0$. But in such a cycle $(z_b, z_0)$ is not a chord. Therefore, $z_a \neq z_0$. Similarly, $z_b \neq z_0$. 
We know that $z_{a+1} \neq z_b$, $z_{a+1} <_{z_a} z_b$, and $z_{a+1} \in V(Q_{i}) \cup V(Q_{i+1}) \cup V(Q_{i+2})$. Since $j-i \leq 2$, $z_b \in V(Q_{i}) \cup V(Q_{i+1}) \cup V(Q_{i+2})$. If $z_b <_{z_a} z_0$, then it contradicts the fact that $z_{a+1}$ is the farthest neighbor of $z_a$ before $z_0$. Therefore, $z_0 <_{z_a} z_b$. Then, either $z_b = z_1$ or $z_1 <_{z_a} z_b$. Recall that $l = \max_{l'}(H' \mbox{has an }l'\mbox{-edge})$, and $(z_0,z_1)$ is an $l$-edge with $z_0 \in V(Q_0)$ and $z_1 \in V(Q_{l})$. Since (i) $(z_a,z_b)$ is a $(j-i)$-edge, where $j-i \leq l$, (ii) $z_0 <_{z_a} z_b$, and (iii) $z_b =z_1$ or $z_1 <_{z_a} z_b$, we have $l \geq j-i = clock\_dist(z_a, z_b) \geq clock\_dist(z_0, z_b) \geq clock\_dist(z_0, z_1) = l$. Hence, $j-i = l$ and $(z_a,z_b)$ is an $l$-edge. We know that $(z_0,z_1)$ is also an $l$-edge with $z_0 \in V(Q_0)$ and $z_1 \in V(Q_l)$. Since $z_0 <_{z_a} z_b$ and $z_b = z_1$ or $z_1 <_{z_0} z_b$, we get $z_a \in V(Q_0)$ and $z_b \in V(Q_l)$. 
From Step 2 of the algorithm we know that $z_0$ is the first  vertex (in a clockwise scan) in $Q_0$ which has an $l$-edge to a vertex in $Q_l$. This implies that, since $z_0 <_{z_a} z_b$,  $z_a = z_0$ which is a contradiction. Hence we prove the claim.
\end{proof}

What is left now is to show that $q+1 \geq p$, i.e., $||C'|| \geq ||C||$, where $C' = z_0 \ldots z_qz_0$ and $C = u_0 \ldots u_{p-1}u_0$. In order to show this, we state and prove the following claims. 
\begin{claim}
\label{cycleSizeClaim}
For every $j \in \{0, \ldots , p-1\}$, $(V(Q_j) \cup V(Q_{j+1}))\cap V(C') \neq \emptyset$.  
\end{claim}
\begin{proof}
Suppose the claim is not true. Find the minimum $j$ that violates the claim. Clearly, $j \neq 0$ as $z_0 \in V(Q_0)$. We claim that $z_q \in V(Q_{j-1})$. Suppose $z_q \notin V(Q_{j-1})$. Let $a = \max \{i~|~z_i \in V(Q_{j-1})\}$ (note that, since $j \neq 0$, by the minimality of $j$, $(V(Q_{j-1}) \cup V(Q_j)) \cap V(C') \neq  \emptyset$ and therefore $V(Q_{j-1}) \cap V(C') \neq \emptyset$). Since $z_a \neq z_q$, by the maximality of $a$, we have $z_{a+1} \notin V(Q_{j-1})$. From our assumption, $(V(Q_j) \cup V(Q_{j+1}))\cap V(C') = \emptyset$ and therefore $z_{a+1} \notin V(Q_{j-1}) \cup V(Q_j) \cup V(Q_{j+1})$. Thus $z_a \neq z_q$ and $z_{a+1}$ is not the farthest neighbor of $z_a$ before $z_0$. This is a contradiction to the way $z_{a+1}$ is chosen by Algorithm $3.1$. 
Hence, $z_q \in V(Q_{j-1})$. We know that $(z_q, z_0) \in E(H')$ with $z_q \in V(Q_{j-1})$ and $z_0 \in V(Q_0)$. Since $l = \max_{l'}(H'\mbox{ has an }l'\mbox{-edge})$, we have $\min(clock\_dist(z_q,z_0),$ $clock\_dist(z_0,z_q)) \leq l$ . That is,  $j \geq p + 1-l$ or $j \leq 1+l$. As $l \leq 2$ (by Claim \ref{Ledgeclaim}), we have $j=p-1$ or $j \leq 1+l$. Since $z_0 \in V(Q_0)$, $(V(Q_{p-1}) \cup V(Q_{0}))\cap V(C') \neq \emptyset$ and hence $j \neq p-1$. Therefore, $j \leq 1+l$.  Since $z_0 \in V(Q_0)$ and $z_1 \in V(Q_{l})$ (recall $l \leq 2$), we get $j = 1 + l$. 
We know that, for every $z_a, z_b \in V(C')$, if $a<b$ then $z_a <_{z_0} z_b$. Therefore, $z_1 <_{z_0} z_q$. We have $z_1 \in V(Q_l)$. Since $j = 1+l$, we also have $z_q \in V(Q_l)$. Thus, we have $z_1, z_q \in V(Q_l)$ and $z_1 <_{z_0} z_q$. But this contradicts the fact that $z_1$ is the last vertex in $Q_l$ encountered in a clockwise scan that has $z_0$ as its neighbor. 
\end{proof}
\begin{figure}
\centerline{\input{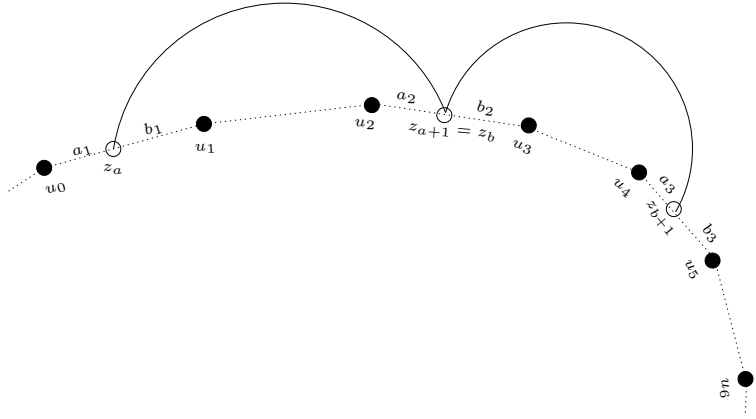}}
\caption{Figure illustrates the case when path $P$ defined in Claim \ref{0-edgeClaim} is a trivial path. The dotted lines between each $u_{i-1}$ and $u_i$ indicate the path $Q_i'$. Each continuous arc corresponds to an edge in the cycle $C' = z_0\ldots z_qz_0$.}
\label{0-edgeFig1}
\end{figure}

\begin{figure}
\centerline{\input{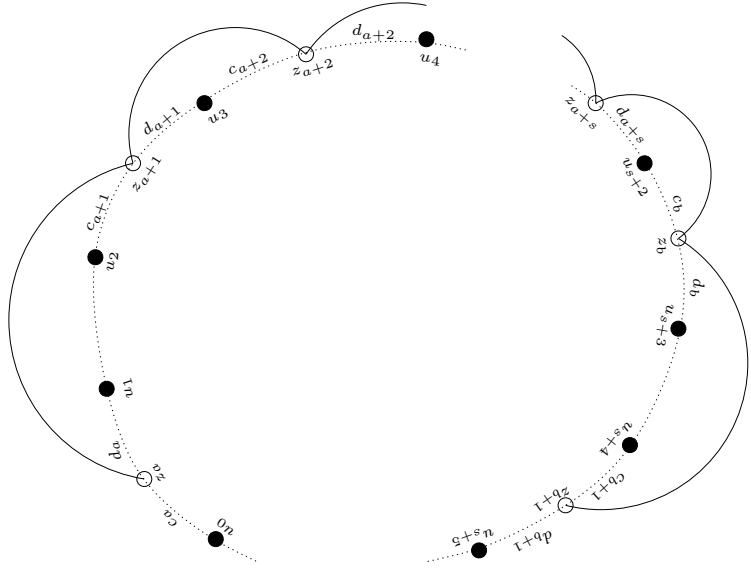}}
\caption{Figure illustrates the case when path $P$ defined in Claim \ref{0-edgeClaim} is 
$P = z_{a+1} z_{a+2}\ldots z_{a+1+s}$, where $s \geq 1$ and $z_{a+1+s} = z_b$. The dotted lines between each $u_{i-1}$ and $u_i$ indicate the path $Q_i'$. Each continuous arc corresponds to an edge in the cycle $C' = z_0\ldots z_qz_0$.}
\label{0-edgeFig2}
\end{figure}

\begin{claim}
\label{0-edgeClaim}
Let $(z_a,z_{a+1}), (z_b,z_{b+1}) \in E(C')$ be two $2$-edges, where $a < b$. Let $P$, $P'$ denote the clockwise $z_{a+1}-z_b$, $z_{b+1}-z_a$ paths respectively in $C'$.   
Both $P$ and $P'$ contain at least one $0$-edge.  
\end{claim}
\begin{proof}
Consider the path $P$ (proof is similar in the case of path $P'$). Path $P$ is a non-trivial path only if $z_{a+1} \neq z_b$. Suppose $z_{a+1} = z_b$ (see Fig. \ref{0-edgeFig1}). Let $z_a \in V(Q_f)$. 
For the sake of ease of notation, assume $f=1$ (the same proof works for any value of $f$). Let $a_1 = ||u_0Q_1'z_a||$, $b_1 = ||z_aQ_1'u_1||$, $a_2 = ||u_2Q_3'z_b||$, $b_2 = ||z_bQ_3'u_3||$, $a_3 = ||u_4Q_5'z_{b+1}||$, and $b_3 = ||z_{b+1}Q_5'u_5||$. 
We know that, for every $w_1, w_2 \in V(H'^{[m]})$ with $w_1 \in A$ and $w_2 \in B$, if $(w_1,w_2) \notin E(H'^{[m]})$ then $d_{H'}(w_1,w_2) \geq m+2$.
Since $(u_0,u_3) \notin E(H'^{[m]}),~(u_1,u_4) \notin E(H'^{[m]})$ and $(u_2,u_5) \notin E(H'^{[m]})$,  we have $a_1 + b_2 \geq m+1$, $b_1 + a_3 \geq m$, and $a_2 + b_3 \geq m+1$.  
Adding the three inequalities and by applying an easy averaging argument we can infer that either $a_1 + b_1 = ||Q_1|| > m$, $a_2 + b_2 = ||Q_{3}|| > m$, or $a_3 + b_3 =||Q_{5}|| > m$ which is a contradiction. Therefore $P$ is a non-trivial path i.e., $z_{a+1} \neq z_b$. Assume $P$ does not contain any $0$-edge. Let $P = z_{a+1} z_{a+2}\ldots z_{a+1+s}$, where $s \geq 1$, $a+1+s = b$, and $(z_{a+1},z_{a+2}) \ldots (z_{a+s},z_{a+1+s})$ are $1$-edges (see Fig. \ref{0-edgeFig2}). Since $(u_0,u_3) \notin E(H'^{[m]}),~(u_1,u_4) \notin E(H'^{[m]})$, we have $c_a+d_{a+1} \geq m+1$ and $d_{a} + d_{a+2} \geq m$ (please refer Fig. \ref{0-edgeFig2} for knowing what $c_a, d_a, \ldots ,c_{b+1},d_{b+1}$ are). Summing up the two inequalities, we get $d_{a+1} + d_{a+2} \geq 2m+1 - (c_{a} + d_a)$. We know that, for each $i \in \{0, \ldots p-1\}$, $||Q_i'|| \leq m$. Therefore, we have $c_a + d_a \leq m$. Hence, $d_{a+1} + d_{a+2} \geq m+1$. Since $(c_{a+1} + d_{a+1}) + (c_{a+2} + d_{a+2}) \leq 2m$, we get 
\begin{eqnarray}
 \label{c1c2ineq}
c_{a+1} + c_{a+2} & \leq & m-1
\end{eqnarray}

Since $(u_{s+2},u_{s+5}) \notin E(H'^{[m]}),~(u_{s+1},u_{s+4}) \notin E(H'^{[m]})$, we have, 
\begin{eqnarray*}
c_b +  d_{b+1} & \geq & m+1 \\
c_{a+s} + c_{b+1} & \geq & m
\end{eqnarray*}
Summing up the two inequalities, we get 
\begin{eqnarray*}
 c_b + c_{a+s} & \geq & 2m+1 - (c_{b+1} + d_{b+1})
\end{eqnarray*}
Since $b=a+s+1$ and  $c_{b+1} + d_{b+1} \leq m$, we get 
\begin{eqnarray}
 \label{cscs+1ineq}
c_{a+s+1} + c_{a+s} & \geq & m+1
\end{eqnarray}
 Substituting for $s=1$ in Inequality \ref{cscs+1ineq}, we get $c_{a+2} + c_{a+1} \geq m+1$. But this contradicts Inequality \ref{c1c2ineq}. Hence $s>1$. Suppose $s=2$. Since $(u_{2},u_{5}) \notin E(H'^{[m]}))$, we have $c_{a+1} + d_{a+3} \geq m$. Adding this with Inequality \ref{cscs+1ineq}, we get $c_{a+1} + c_{a+2} \geq (2m+1) - (c_{a+3} + d_{a+3}) \geq m+1$. But this contradicts Inequality \ref{c1c2ineq}. Hence $s>2$. Since $(u_{s},u_{s+3}) \notin E(H'^{[m]})), \ldots ,(u_{2},u_{5}) \notin E(H'^{[m]}))$, we have the following inequalities:-
\begin{eqnarray*}
c_{a+s-1} + d_{a+s+1} & \geq & m \\
\vdots & \vdots & \vdots \\
c_{a+1} + d_{a+3} & \geq & m
\end{eqnarray*}
Adding the above set of inequalities and applying the fact that $c_i + d_i \leq m $, $\forall i \in \{0, \ldots q\}$, we get $c_{a+1} + c_{a+2} + d_{a+s} + d_{a+s+1} \geq 2m$. Adding this with Inequality \ref{cscs+1ineq}, we get $c_{a+1} + c_{a+2} \geq (3m+1) - (c_{a+s+1} + d_{a+s+1}) - (c_{a+s} + d_{a+s})  \geq m+1$. But this contradicts Inequality \ref{c1c2ineq}. Hence we prove the claim. 
\end{proof}

\begin{claim}
\label{pathContribClaim}
For every $j,j' \in \{0, \ldots, p-1\}$, where $j < j'$ and $(V(Q_j) \cup V(Q_{j'}))\cap V(C') = \emptyset$, there exist $i, i' \in \{0, \ldots, p-1\}$, where only $i$ satisfies $j < i < j'$, such that $|V(Q_{i})\cap V(C')| \geq 2$ and $|V(Q_{i'})\cap V(C')| \geq 2$.  
\end{claim}
\begin{proof}
By Claim \ref{cycleSizeClaim}, (i) $j' \neq j+1$ or $j' \neq j-1$, and  (ii) there exist $r,r' \in \{0, \ldots , q\}$ such that $(z_{r}, z_{r+1})$ is a $2$-edge with its endpoints on $Q_{j-1}$ and $Q_{j+1}$ and  $(z_{r'}, z_{r'+1})$ is a $2$-edge with its endpoints on $Q_{j'-1}$ and $Q_{j'+1}$. By Claim \ref{0-edgeClaim}, we know that if $P$, $P'$ denote the clockwise $z_{r+1}- z_{r'}$, $z_{r'+1}-z_r$ paths respectively in $C'$, then both $P$ and $P'$ contains at least one $0$-edge. This proves the claim. 
\end{proof}

In order to show that the size of cycle $C'$ ($= z_0 \ldots z_qz_0$) is at least $p$, we consider the following three cases:- \vspace{0.05in} \\ 
\emph{case $|\{Q_j \in \{Q_0 \ldots Q_{p-1}\}~|~V(Q_j) \cap V(C') = \emptyset \}| = 0$}: In this case, for every $j \in \{0, \ldots p-1\}$, $Q_j$ contributes to $V(C')$ and therefore $||C'|| \geq p = ||C||$. \vspace{0.05in} \\
\emph{case $|\{Q_j \in \{Q_0 \ldots Q_{p-1}\}~|~V(Q_j) \cap V(C') = \emptyset \}| = 1$}: Let $Q_j$ be that only path (among $Q_0 \ldots Q_{p-1}$) that does not contribute to $V(C')$. Then we claim that there exists a $Q_{j'}$, where $j' \neq j$, such that  $V(C') \cap V(Q_{j'}) \geq 2$. Suppose the claim is not true then it is easy to see that $||C'|| = p-1$ which is an odd number thus contradicting the bipartitedness of $H'$. Hence the claim is true. Now, by applying the claim it is easy to see that $||C'|| = \sum_{j}|V(C') \cap V(Q_j)| \geq p = ||C||$. \vspace{0.05in} \\
\emph{case $|\{Q_j \in \{Q_0 \ldots Q_{p-1}\}~|~V(Q_j) \cap V(C') = \emptyset \}| > 1$}: Scan vertices of $H'$ starting from any vertex in clockwise direction. Claim \ref{pathContribClaim} ensures that between every $Q_j$ and $Q_{j'}$, which do not contribute to $V(C')$, encountered there exists a $Q_i$ which compensates by contributing at least two vertices to $V(C')$. Therefore, $||C'|| \geq p = ||C||$. 
\hfill \bbox
\end{proof}

\end{document}